\documentclass[12pt,a4paper]{amsart}
\usepackage{amsmath}
\usepackage[english]{babel}
\usepackage{verbatim}
\usepackage[shortlabels]{enumitem}
\usepackage{tikz-cd}

\newtheoremstyle{mio}%
	{}{} 
	{\itshape}{} 
	{\bfseries}{.}{ } 
	{#1 #2\thmnote{~\mdseries(#3)}} 
\theoremstyle{mio}
\newtheorem{teor}{Theorem}[section]
\newtheorem{cor}[teor]{Corollary}
\newtheorem{prop}[teor]{Proposition}
\newtheorem{lemma}[teor]{Lemma}
\newtheorem{defin}[teor]{Definition}

\newtheoremstyle{definition2}%
	{}{} 
	{}{} 
	{\bfseries}{.}{ } 
	{#1 #2\thmnote{\mdseries~ #3}} 
\theoremstyle{definition2}

\DeclareMathOperator{\Max}{Max}
\DeclareMathOperator{\rad}{rad}
\newcommand{\ins}[1]{\mathbb{#1}}
\newcommand{\insN}{\ins{N}}
\newcommand{\insZ}{\ins{Z}}
\newcommand{\insQ}{\ins{Q}}

\newcommand{\inN}{\in\insN}

\newcommand{\inZ}{\in\insZ}

\newcommand{\filtro}{\mathcal{F}}
\newcommand{\insfiltri}{\boldsymbol{\mathcal{F}}}
\newcommand{\insV}{\boldsymbol{\mathcal{V}}}
\newcommand{\nz}{\bullet}

\newcommand{\Cl}{\mathrm{Cl}}

\newcommand{\pow}[1]{\mathrm{pow}({#1})}
\newcommand{\setclos}[1]{\mathcal{X}(#1)}
\newcommand{\setdiv}[1]{\mathcal{D}(#1)}

\newcommand{\arp}{\arrow[dashed]{urr}}
\newcommand{\arpp}{\arrow[dashed]{urrrr}}

\title[The Golomb topology and units of quotients]{The Golomb topology on a Dedekind domain and the group of units of its quotients}
\author{Dario Spirito}
\date{\today}
\address{Dipartimento di Matematica e Fisica, Universit\`a degli Studi ``Roma Tre'', Roma, Italy}
\email{spirito@mat.uniroma3.it}

\begin{document}

\begin{abstract}
We study the Golomb spaces of Dedekind domains with torsion class group. In particular, we show that a homeomorphism between two such spaces sends prime ideals into prime ideals and preserves the $P$-adic topology on $R\setminus P$. Under certain hypothesis, we show that we can associate to a prime ideal $P$ of $R$ a partially ordered set, constructed from some subgroups of the group of units of $R/P^n$, which is invariant under homeomorphisms, and use this result to show that the unique self-homeomorphisms of the Golomb space of $\insZ$ are the identity and the multiplication by $-1$. We also show that the Golomb space of any Dedekind domain contained in the algebraic closure of $\insQ$ is not homeomorphic to the Golomb space of $\insZ$.
\end{abstract}

\maketitle

\section{Introduction}
Let $R$ be an integral domain. The \emph{Golomb topology} of $R$ is the topology on $R^\nz:=R\setminus\{0\}$ generated by the coprime cosets; we denote by $G(R)$ the space $R^\nz$ endowed with the this topology, and call it the \emph{Golomb space} of $R$. The Golomb topology on the set $\insZ^+$ of positive integer was introduced by Brown \cite{brown-golomb} and subsequently studied by Golomb \cite{golomb-connectedtop,golomb-aritmtop}. On general domains, the Golomb topology was considered alongside several other coset topologies (see for example \cite{knopf}), and was shown to provide a way to generalize Furstenberg's ``topological'' proof of the infinitude of primes in a more general context \cite{furstenberg,clark-euclidean}. See \cite[Section 4]{clark-golomb} for a more detailed historical overview of the subject.

Two recent articles have shed more light on the Golomb topology. The first one, due to Banakh, Mioduszewski and Turek \cite{bmt-golomb}, deals with the ``classical'' subject of the Golomb topology on $\insZ^+$, with the explicit goal of deciding if this space is \emph{rigid}, i.e., if it does not admit any self-homeomorphism; in particular, they show that any self-homeomorphism of this space fixes 1 \cite[Theorem 5.1]{bmt-golomb}. The second one, due to Clark, Lebowitz-Lockard and Pollack \cite{clark-golomb}, studies Golomb spaces on general domains, in particular when the ring $R$ is a Dedekind domain with infinitely many maximal ideals: under this hypothesis, they show that $G(R)$ is a Hausdorff space that is not regular, and that it is a connected space that is totally disconnected at each of its points. They also raise the \emph{isomorphism problem}: can two nonisomorphic Dedekind domains with infinitely many maximal ideals (or, more generally, two integral domains with zero Jacobson radical) have homeomorphic Golomb spaces? As a first step in this study, they prove that any homeomorphism of Golomb topologies sends units to units \cite[Theorem 13]{clark-golomb}, and thus that two domains with a different number of units have nonhomeomorphic Golomb spaces. We note that the rigidity problem and the isomorphism problem can be unified into a single question:

\smallskip

\textbf{Problem.} Let $R,S$ be two Dedekind domains with infinitely many maximal ideals, and let $h:G(R)\longrightarrow G(S)$ be a homeomorphism. Is it true that there is a ring isomorphism $\sigma:R\longrightarrow S$ and a unit $u\in S$ such that $h(x)=u\sigma(x)$ for all $x\in R$?

\smallskip

In this paper, we show that the only self-homeomorphisms of the Golomb space $G(\insZ)$ are the identity and the multiplication by $-1$ (Theorem \ref{teor:selfOmefZ}), and that if $R$ is a Dedekind domain contained in the algebraic closure $\overline{\insQ}$ of $\insQ$ such that $G(\insZ)\simeq G(R)$ then $R=\insZ$, thus giving a complete answer to the above question for $R=S=\insZ$ and a partial answer for $R=\insZ$. While the method we use works best for the ring of integers, we work as much as possible in a greater generality: the main restrictions we have to put (especially in Sections \ref{sect:clos} and \ref{sect:Yk}) are that the class group of Dedekind domain we consider must be torsion, and that some quotients of the group of units of $R/P^n$ are cyclic.

The structure of the paper is as follows. In Section \ref{sect:radical}, we generalize \cite[Lemma 5.6]{bmt-golomb} to the case of general Dedekind domains; in particular, we show that the partially ordered set $\insV(R)$ formed by the subsets of $\Max(R)$ that can be written as $V(x):=\{M\in\Max(R)\mid x\in M\}$ for some $x\in R^\nz$ is a topological invariant of the Golomb topology (Proposition \ref{prop:isoV}). Through this result, we prove that if $G(R)\simeq G(S)$ then the class groups of $R$ and $S$ are either both torsion or both non-torsion (Theorem \ref{teor:torsion}) and, if they are torsion, then a homeomorphism between $G(R)$ and $G(S)$ sends prime ideals to prime ideals and radical ideals to radical ideals.

In Section \ref{sect:Ptop}, given a prime ideal $P$ of $R$, we show how to construct from the Golomb topology a new topology on $R\setminus P$ (the \emph{$P$-topology}), which allows to concentrate on the cosets in the form $a+P^n$. Section \ref{sect:HnP} collects some results about the groups $H_n(P):=U(R/P^n)/\pi_n(U(R))$.

In Section \ref{sect:clos}, we study the sets $\pow{a}:=\{ua^n\mid u\in U(R)\}$ of powers of the elements $a\in R\setminus P$, and in particular their closure in the $P$-topology. We relate this closure to the cyclic subgroups of the groups $H_n(P)$; in particular, we show that under some hypothesis (among which that $R$ has torsion class group and that the $H_n(P)$ are cyclic) the closure of $\pow{a}$ is characterized by the index of the subgroup generated by $a$ in $H_n(P)$ for large $n$. Restricting to almost prime elements (i.e., irreducible elements generating a primary ideal) we show that there is a bijective correspondence between these closures and a set of integers depending on the cardinality of the $H_n(P)$ (Theorem \ref{teor:corresp-Theta}), and that this structure is preserved under homeomorphisms of Golomb spaces (Propositions \ref{prop:omef-XP} and \ref{prop:DP}). In Section \ref{sect:Yk}, we explicit enough of this correspondence to characterize completely the self-homeomorphisms of $G(\insZ)$ (Theorem \ref{teor:selfOmefZ}).

\section{Radical and prime ideals}\label{sect:radical}
Throughout the paper, $R$ will a Dedekind domain, i.e., a commutative unitary ring with no zerodivisors such that every ideal can be written as a product of prime ideals; equivalently, such that $R$ is a one-dimensional Noetherian integrally closed domain. We denote by $\Max(R)$ the set of maximal elements of $R$; if $x\in R\setminus\{0\}$, we set $V(x):=\{M\in\Max(R)\mid x\in M\}$. We denote by $U(R)$ the set of units of $R$, and by $\rad(I)$ the radical of the ideal $I$.

The \emph{class group} $\Cl(R)$ of $R$ is the abelian group of the fractional ideals of $R$ modulo principal ideals, with the operation being the multiplication between ideals. The class group of $R$ is torsion if and only if every prime ideal contains a principal primary ideal \cite[Proposition 3.1]{gilmer_qr}.

For every subset $I\subseteq R$, we set $I^\nz:=I\setminus\{0\}$. The \emph{Golomb space} $G(R)$ is the topological space on $R^\nz$ generated by the cosets $a+I$ such that $\langle a,I\rangle=R$. (Note that if $\langle a,I\rangle=R$ then $0\notin a+I$ and thus $a+I\subseteq R^\nz$.) For $X\subseteq R^\nz$, we denote by $\overline{X}$ the closure of $X$ is the Golomb topology.

The closure of the coprime cosets can be completely described.
\begin{lemma}\label{lemma:chiusura}
\cite[Lemma 15]{clark-golomb} Let $R$ be a Dedekind domain, let $I$ be a nonzero ideal of $R$ and let $x\in R^\nz$ be such that $\langle x,I\rangle=R$. Let $I=P_1^{e_1}\cdots P_n^{e_n}$ be the factorization of $I$. Then,
\begin{equation*}
\overline{x+I}=\left(\bigcap_{i=1}^nP_i\cup(x+P_i^{e_i})\right)^\nz.
\end{equation*}
\end{lemma}

In particular, we immediately obtain the following.
\begin{cor}\label{cor:chiusprod}
Let $R$ be a Dedekind domain, and let $I,J$ be coprime ideals. For every $x$ such that $\langle x,I\rangle=\langle x,J\rangle=R$, we have $\overline{x+IJ}=\overline{x+I}\cap\overline{x+J}$.
\end{cor}

The purpose of this section is to generalize the results obtained in \cite[Section 5]{bmt-golomb} on the relationship between the Golomb topology and the prime divisors of an element $x\in R^\nz$. Following the methods used therein, we define $\filtro_x$ as the set of all $F\subseteq R^\nz$ such that there are a neighborhood $U_x$ of $x$ and a neighborhood $U_1$ of $1$ such that $\overline{U_x}\cap\overline{U_1}\subseteq F$. 

Part \ref{prop:Fx:prime} of the following proposition corresponds to \cite[Lemma 5.5(a)]{bmt-golomb}, while part \ref{prop:Fx:cont} corresponds to \cite[Lemma 5.6]{bmt-golomb}.
\begin{prop}\label{prop:Fx}
Let $R$ be a Dedekind domain. Let $x,y\in R^\nz$ and let $M\in\Max(R)$. Then, the following hold.
\begin{enumerate}[(a)]
\item\label{prop:Fx:filtro} $\filtro_x$ is a filter.
\item\label{prop:Fx:prime} $M^\nz\in\filtro_x$ if and only if $x\notin M$.
\item\label{prop:Fx:cont} $\filtro_x\subseteq\filtro_y$ if and only if $V(y)\subseteq V(x)$.
\end{enumerate}
\end{prop}
\begin{proof}
\ref{prop:Fx:filtro} By the proof of \cite[Theorem 8(a)]{clark-golomb} (and the discussion in Section 3 therein), for every open sets $V_1,\ldots,V_n$ the intersection $\overline{V_1}\cap\cdots\cap\overline{V_n}$ is nonempty; the claim follows.

\ref{prop:Fx:prime} is a direct consequence of \cite[Lemma 17]{clark-golomb}, applied with $y=1$.

\ref{prop:Fx:cont} Suppose $\filtro_x\subseteq\filtro_y$, and let $P\in V(y)$. Then, $y\in P$, so by point \ref{prop:Fx:prime} $P\notin\filtro_y$; hence, $P\notin\filtro_x$ and thus again $x\in P$, i.e., $P\in V_x$.

Conversely, suppose $V(y)\subseteq V(x)$. Let $F\in\filtro_x$; then, there are ideals $I,J$ of $R$ such that $\langle x,I\rangle=R$ and such that $\overline{x+I}\cap\overline{1+J}\subseteq F$. Without loss of generality, we can suppose that $J\subseteq I$ and that $J=IJ'$ for some $J'$ such that $\langle I,J'\rangle=R$. Let $I=\prod_iP_i^{e_i}$ be the prime decomposition of $I$; by Corollary \ref{cor:chiusprod}, we have
\begin{align*}
\overline{x+I}\cap\overline{1+J}= & \overline{x+I}\cap\overline{1+I}\cap\overline{1+J'}=\\
=& \bigcap_i\overline{x+P_i^{e_i}}\cap\overline{1+P_i^{e_i}}\cap\overline{1+J'}.
\end{align*}
For each $i$, let $n_i$ be an integer such that $y-1\notin P_i^{n_ie_i}$. Then, by Lemma \ref{lemma:chiusura},
\begin{equation*}
\overline{y+P_i^{n_ie_i}}\cap\overline{1+P_i^{n_ie_i}}=((y+P_i^{n_ie_i})\cup P_i)^\nz\cap((1+P_i^{n_ie_i})\cup P_i)^\nz=P_i^\nz.
\end{equation*}
Let $I':=\prod_iP_i^{e_in_i}$: then,
\begin{align*}
\overline{y+I'}\cap\overline{1+I'}=& \bigcap_i\overline{y+P_i^{n_ie_i}}\cap\overline{1+P_i^{n_ie_i}}=\left(\bigcap_iP_i\right)^\nz\subseteq\\
\subseteq &\bigcap_i\overline{x+P_i^{e_i}}\cap\overline{1+P_i^{e_i}}=\\
= & \overline{x+I'}\cap\overline{1+I'}\subseteq\overline{x+I}\cap\overline{1+I},
\end{align*}
and thus
\begin{equation*}
\overline{y+I'}\cap\overline{1+I'J'}=\overline{y+I'}\cap\overline{1+I'}\cap\overline{1+J'}\subseteq\overline{x+I}\cap\overline{1+I}\cap\overline{1+J'}\subseteq F.
\end{equation*}
Since the radical of $I$ and $I'$ is the same and $\langle x,I\rangle=R$, also $\langle x,I'\rangle=R$; since $V(y)\subseteq V(x)$, we have $\langle y,I'\rangle=R$, and thus $y+I'$ is an open neighborhood of $y$. Hence, $F\in\filtro_y$ and thus $\filtro_x\subseteq\filtro_y$, as claimed.
\end{proof}

Let $R$ be a Dedekind domain. We consider two sets associated to $R$:
\begin{equation*}
\insfiltri(R):=\{\filtro_x\mid x\in R^\nz\}
\end{equation*}
and
\begin{equation*}
\insV(R):=\{V(x)\mid x\in R^\nz\}.
\end{equation*}
The previous proposition establishes a relation between them.
\begin{prop}\label{prop:filtriV}
Let $R$ be a Dedekind domain. The map
\begin{equation*}
\begin{aligned}
\Psi\colon\insfiltri(R) & \longrightarrow\insV(R),\\
\filtro_x & \longmapsto V(x)
\end{aligned}
\end{equation*}
is well-defined and an anti-isomorphism (when $\insfiltri(R)$ and $\insV(R)$ are endowed with the containment order).
\end{prop}
\begin{proof}
Proposition \ref{prop:Fx}\ref{prop:Fx:cont} guarantees that $\Psi$ is well-defined, injective and order-reversing, while the surjectivity is obvious.
\end{proof}

\begin{prop}\label{prop:isoV}
Let $R,S$ be Dedekind domains and $h:G(R)\longrightarrow G(S)$ be a homeomorphism. Then, the following hold.
\begin{enumerate}[(a)]
\item\label{prop:isoV:filtri} If $h(1)=1$, then $h(\filtro_x)=\filtro_{h(x)}$ for every $x\in R^\nz$.
\item\label{prop:isoV:V} $h$ induces an order isomorphism
\begin{equation*}
\begin{aligned}
\overline{h}\colon\insV(R) & \longrightarrow\insV(S),\\
V(x) & \longmapsto V(h(x)).
\end{aligned}
\end{equation*}
\end{enumerate}
\end{prop}
\begin{proof}
\ref{prop:isoV:filtri} Since $h$ is a homeomorphism and $h(1)=1$, $h$ sends neighborhoods of $x$ into neighborhoods of $h(x)$, and neighborhoods of $1$ into neighborhoods of $1$, and analogously for their closures. The claim follows by the definition of $\filtro_x$.

\ref{prop:isoV:V} For every unit $v$ of $S$, let $\psi_v:G(S)\longrightarrow G(S)$ be the multiplication by $v$. Clearly, $\psi_v$ is a self-homeomorphism of $G(S)$.

Let $u:=h(1)$. By \cite[Theorem 13]{clark-golomb}, $u$ is a unit of $S$, and thus $\psi_u$ is a self-homeomorphism of $G(S)$. Then, $h=\psi_u\circ\psi_{u^{-1}}\circ h$; setting $h':=\psi_{u^{-1}}\circ h$, it is enough to show the claim separately for $\psi_u$ and for $h'$.

For every $y\in S^\nz$, $V(uy)=V(y)$; hence, the map
\begin{equation*}
\begin{aligned}
\widetilde{\psi_u}\colon\insfiltri(S) & \longrightarrow\insfiltri(S),\\
\filtro_x & \longmapsto \filtro_{ux}
\end{aligned}
\end{equation*}
is the identity, and in particular it is an order isomorphism. Then, if $\Psi$ is the map of Proposition \ref{prop:filtriV}, we have that $\Psi\circ\widetilde{\psi_u}\circ\Psi^{-1}$ is an order-isomorphism of $\insV(S)$ with itself; unraveling the definition we see that $\overline{\psi_u}=\Psi\circ\widetilde{\psi_u}\circ\Psi^{-1}$, and the claim is proved.

Consider now $h'$. Then, $h'(1)=u^{-1}h(1)=1$. By the previous point, $h'(\filtro_x)=\filtro_{h'(x)}$; hence, by Proposition \ref{prop:Fx}\ref{prop:Fx:cont}, the map
\begin{equation*}
\begin{aligned}
\widetilde{h'}\colon\insfiltri(R) & \longrightarrow\insfiltri(R),\\
\filtro_x & \longmapsto \filtro_{h'(x)}
\end{aligned}
\end{equation*}
is well-defined and an order-isomorphism. As before, we see that $\overline{h'}=\Psi\circ\widetilde{h'}\circ\Psi^{-1}$ and that the right hand side is an order-isomorphism between $\insV(R)$ and $\insV(S)$, and the claim is proved.
\end{proof}

Since a Dedekind domain is locally finite, $\insV(R)$ is always a subset of $\mathcal{P}_{\mathrm{fin}}(\Max(R))$, the set of finite subsets of $\Max(R)$. When the class group of $R$ is torsion, we have equality: indeed, if $P$ is a maximal ideal then $P^k$ is principal for some $k$, and thus there is an $x_P$ (for example, the generator of $P^k$) such that $V(x_P)=\{P\}$. Hence, $\{P_1,\ldots,P_n\}$ is just equal to $V(x_{P_1}\cdots x_{P_n})$. This is actually an equivalence: indeed, if the class of $P$ in $\Cl(R)$ is not torsion then $V(x)\neq\{P\}$ for every $x\in R^\nz$. We can upgrade this difference.
\begin{teor}\label{teor:torsion}
Let $R,S$ be Dedekind domains such that $G(R)$ and $G(S)$ are isomorphic. Then, the class group of $R$ is torsion if and only if the class group of $S$ is torsion.
\end{teor}
\begin{proof}
Suppose that the class group of $R$ is torsion while the class group of $S$ is not, and let $\mathcal{M}(R)$ (respectively, $\mathcal{M}(S)$) be the set of maximal elements of $\insV(R)$ (resp., $\insV(S)$).

Every element of $\mathcal{M}(R)$ is a singleton. Therefore, if $\Delta\subseteq\mathcal{M}(R)$ is finite, say $\Delta=\{\{P_1\},\ldots,\{P_n\}\}$, then $\inf\Delta$ exists and is equal to $\{P_1,\ldots,P_n\}$. Therefore, $\inf\Delta\neq\inf\Lambda$ for every finite $\Delta\neq\Lambda$.

We claim that this does not hold in $\insV(S)$. Indeed, since the class group of $S$ is not torsion there is a maximal ideal $P$ such that no power of $P$ is principal. Let $x\in P\setminus P^2$: then, $xR=PA$ for some ideal $A$ coprime with $P$. By prime avoidance, we can find an $y\in P\setminus P^2$ that is not contained in any prime ideal containing $A$; then, $y=PB$ for some $B$ coprime with $P$ and $A$. In particular, the classes of $A$ and $B$ in the class group are the same (they are both the inverse of the class of $P$). Let $H$ be a nonzero ideal in the same class of $P$ that is coprime with $P$, $A$ and $B$ (which exists by prime avoidance); then, $HA$ and $HB$ are principal, say $HA=zR$ and $HB=wR$. Then, $xwR=PAHB=PBHA=yzR$, and in particular $V(xw)=V(yz)$. Let $\mathcal{M}(x)$ be the set of maximal elements of $\insV(R)$ containing $V(x)$, and likewise define $\mathcal{M}(y)$, $\mathcal{M}(z)$ and $\mathcal{M}(w)$; then, $\inf(\mathcal{M}(x)\cup\mathcal{M}(w))=V(x)\cup V(w)=V(xw)=V(yz)=\inf(\mathcal{M}(y)\cup\mathcal{M}(z))$. We claim that $\mathcal{M}(x)\cup\mathcal{M}(w)\neq\mathcal{M}(y)\cup\mathcal{M}(z)$.

There is an element of $\mathcal{M}(x)$ containing $P$: since the class of $P$ is not torsion, it must be equal to $\Theta:=\{P,Q_1,\ldots,Q_n\}$ for some prime ideals $Q_1,\ldots,Q_n$ containing $A$. Since $z\notin P$, no element of $\mathcal{M}(z)$ contains $P$, and in particular $\Theta\notin\mathcal{M}(z)$. If $\Theta'\in\mathcal{M}(y)$ contains $P$ then $\Theta'=\{P,L_1,\ldots,L_m\}$ for some prime ideals $L_1,\ldots,L_m$ containing $B$; since $A$ and $B$ are coprime, each $Q_i$ is different from each $L_j$, and thus $\Theta'\neq\Theta$, and so $\Theta\notin\mathcal{M}(y)$. Hence, there are finite subsets $\Delta\neq\Lambda$ of $\mathcal{M}(S)$ such that $\inf\Delta=\inf\Lambda$; since this property is purely order-theoretic, it follows that $\insV(R)$ and $\insV(S)$ are not isomorphic. By Proposition \ref{prop:isoV}\ref{prop:isoV:filtri}, neither $G(R)$ and $G(S)$ are homeomorphic.
\end{proof}

It would be interesting to know how much further this method can be pushed: for example, is it possible to recover the rank of the class group of $R$ from the order structure of $\insV(R)$?

\medskip

We now consider more in detail the case where the class group of $R$ is torsion. Given $\Delta\subseteq\Max(R)$, we define
\begin{equation*}
G_\Delta(R):=\{x\in R^\nz\mid V(x)=\Delta\}.
\end{equation*}
By the discussion before Theorem \ref{teor:torsion}, if $\Cl(R)$ is torsion then $G_\Delta(R)\neq\emptyset$ for every finite $\Delta\subseteq\Max(R)$.

The following is an analogue of \cite[Lemmas 5.8 and 5.9]{bmt-golomb}.
\begin{prop}\label{prop:GDelta}
Let $R,S$ be Dedekind domains with torsion class group, and let $h:G(R)\longrightarrow G(S)$ be a homeomorphism. Then, there is a bijection $\sigma:\Max(R)\longrightarrow\Max(S)$ such that $h(G_\Delta(R))=G_{\sigma(\Delta)}(S)$.
\end{prop}
\begin{proof}
By \cite[Theorem 13]{clark-golomb}, $h(1)$ is a unit of $S$. The multiplication by $u$ is a homeomorphism of $S$ which sends every $G_\Lambda(S)$ into itself; hence, passing to $h':G(R)\longrightarrow G(S)$, $x\mapsto h(1)^{-1}h(x)$, we can suppose without loss of generality that $h(1)=1$.

We claim that $|V(x)|=|V(h(x))|$ for every $x\in R^\nz$. Indeed, since $\Cl(R)$ is torsion $\insV(R)\simeq\mathcal{P}_{\mathrm{fin}}(\Max(R))$, and thus $|V(x)|$ is equal to $1$ plus the length of an ascending chain of $\insV(R)$ starting from $V(x)$. By Proposition \ref{prop:isoV}\ref{prop:isoV:V}, this property passes to $\insV(S)$, and thus $|V(x)|=|V(h(x))|$.

Let $\overline{\sigma}$ be the restriction to $\mathcal{M}(R)$ (the set of maximal elements of $\insV(R)$) of the isomorphism $\overline{h}$ of Proposition \ref{prop:isoV}\ref{prop:isoV:V}. Since $\mathcal{M}(R)$ is in natural bijective correspondence with $\Max(R)$ (just send $\{P\}$ into $P$) we get a bijection $\sigma:\Max(R)\longrightarrow\Max(S)$, such that if $P\in\Max(R)$ and $xR$ is $P$-primary then $\sigma(P)$ is the unique maximal ideal of $S$ containing $h(x)$.

If now $x\in G_\Delta(R)$, then $\Delta=\{P\in\Max(R)\mid P\notin\filtro_x\}$; hence, $\sigma(\Delta)=\{Q\in\Max(S)\mid Q\notin\filtro_{h(x)}\}$, and thus $h(x)\in G_{\sigma(\Delta)}(S)$, so $h(G_\Delta(R))\subseteq G_{\sigma(\Delta)}(S)$. Applying the same reasoning to $h^{-1}$ gives the opposite inclusion, and thus $h(G_\Delta(R))=G_{\sigma(\Delta)}(S)$.
\end{proof}

If $h:G(R)\longrightarrow G(S)$, we denote by $h_e:R\longrightarrow S$ the extension of $h$ sending $0$ to $0$.
\begin{teor}\label{teor:radical}
Let $R,S$ be Dedekind domains with torsion class group, and let $h:G(R)\longrightarrow G(S)$ be a homeomorphism. Let $I$ be a radical ideal of $R$. Then, the following hold.
\begin{enumerate}[(a)]
\item\label{teor:radical:rad} $h_e(I)$ is a radical ideal of $S$.
\item\label{teor:radical:num} The number of prime ideals of $R$ containing $I$ is equal to the number of prime ideals of $S$ containing $h_e(I)$.
\item\label{teor:radical:primi} If $I$ is prime, $h_e(I)$ is prime.
\end{enumerate}
\end{teor}
\begin{proof}
Since $I$ is radical, $I=\bigcup\{G_\Delta(R)\mid V(I)\supseteq\Delta\}\cup\{0\}$; hence, applying Proposition \ref{prop:GDelta},
\begin{align*}
h_e(I)= & h\left(\bigcup\{G_\Delta(R)\mid V(I)\supseteq\Delta\}\right)\cup\{0\}=\\
=& \bigcup\{h(G_\Delta(R))\mid V(I)\supseteq\Delta\}\cup\{0\}=\\
=&\bigcup\{G_\Lambda(S)\mid V(I)\subseteq\sigma^{-1}(\Lambda)\}\cup\{0\}=\\
=&\bigcup\{G_\Lambda(S)\mid \sigma(V(I))\subseteq\Lambda\}\cup\{0\}=J
\end{align*}
where $J$ is the radical ideal such that $V(J)=\Delta$, i.e., $J=\bigcap_{Q\in\Lambda}Q$. \ref{teor:radical:rad} is proved.

\ref{teor:radical:num} follows from the fact that that the number of prime ideals containing $I$ is the least $n$ such that there is a subset $\Delta\subseteq\Max(R)$ of cardinality $n$ such that $G_\Delta(R)\subseteq I$. \ref{teor:radical:primi} is immediate from \ref{teor:radical:num}.
\end{proof}

\section{The $P$-topology}\label{sect:Ptop}
The Golomb topology on a Dedekind domain $R$ is a very ``global'' structure: that is, it depends at the same time on all the prime ideals of $R$. In this section, we show a way to ``isolate'' the neighborhoods relative to a single prime ideal $P$, i.e., in the form $a+P^n$. The main idea is the following.
\begin{prop}\label{prop:clopen}
Let $R$ be a Dedekind domain and let $P$ be a prime ideal of $R$; take $\Omega\subseteq R\setminus P$. If $\Omega$ is clopen in $R\setminus P$, then for every $x\in\Omega$ there is an $n\geq 1$ such that $x+P^n\subseteq\Omega$.
\end{prop}
\begin{proof}
Fix $\Omega$ clopen in $R\setminus P$ and let $x\in\Omega$. Since $R\setminus P$ is open, $\Omega$ is also an open set of $G(R)$, and thus there is an ideal $I$ such that $x+I\subseteq\Omega$; since $(x+I)\cap P=\emptyset$, we can write $I=P^nJ$ for some $n\geq 1$ and some ideal $J$ coprime with $P$. We claim that $x+P^n\subseteq\Omega$.

Otherwise, let $y\in(x+P^n)\setminus\Omega$; then, $y\in R\setminus P$, and since $(R\setminus P)\setminus\Omega$ is clopen in $R\setminus P$ we can find, as in the previous paragraph, an integer $m\geq 1$ and an ideal $L$ coprime with $P$ such that $y+P^mL\subseteq(R\setminus P)\setminus\Omega$. Since $\Omega$ is clopen in $R\setminus P$, we have $\overline{\Omega}\cap(R\setminus P)=\Omega$; hence, $\overline{x+P^nJ}\cap(R\setminus P)\subseteq\Omega$. Likewise, $\overline{y+P^mL}\cap(R\setminus P)\subseteq(R\setminus P)\setminus\Omega$, and thus in particular $\overline{x+P^nJ}\cap\overline{y+P^mL}=\emptyset$. However,
\begin{equation*}
\overline{x+P^nJ}=\overline{x+P^n}\cap\overline{x+J}=((x+P^n)\cup P)^\nz\cap\overline{x+J}\supseteq(x+P^n)\cap\rad(J)^\nz
\end{equation*}
and likewise $\overline{y+P^mL}\supseteq(y+P^m)\cap\rad(L)^\nz$. Since $y\in x+P^n$, the intersection $(x+P^n)\cap(y+P^m)$ is nonempty, and thus it contains a coset $z+P^t$. Since $J$ and $L$ are coprime with $P$, we have $(x+P^t)\cap\rad(J)^\nz\cap\rad(L)^\nz\neq\emptyset$; this contradicts the construction of $J$ and $L$, and thus $y$ cannot exist, i.e., $x+P^n\subseteq\Omega$. The claim is proved.
\end{proof}

\begin{cor}
Let $R,S$ be Dedekind domain with torsion class group, let $h:G(R)\longrightarrow G(S)$ be a homeomorphism, and let $P$ be a prime ideal of $R$. For every $x\in R\setminus P$, there is an $n$ such that $h(x)+h(P)^n\subseteq h(x+P)$.
\end{cor}
\begin{proof}
Since $\overline{x+P}=(x+P)\cup P^\nz$, the set $x+P$ is a clopen set of $R\setminus P$. Hence, $h(x+P)$ is clopen in $S\setminus h(P)$; we now apply the previous proposition.
\end{proof}

Let $P$ be a prime ideal of $R$. We define the \emph{$P$-topology} on $R\setminus P$ as the topology generated by the $\Omega\subseteq R\setminus P$ that are clopen $R\setminus P$, with respect to the Golomb topology. Since every coprime coset $a+P^n$ is clopen in $R\setminus P$, Proposition \ref{prop:clopen} implies that the $P$-topology is generated by $a+P^n$, for $a\in R\setminus P$ and arbitrary $n$. Therefore, the $P$-topology on $R\setminus P$ actually coincides with the restriction of the $P$-adic topology.

In our context, the most useful property of the $P$-topology is that it depends uniquely on the Golomb topology, in the following sense.
\begin{teor}\label{teor:ptop}
Let $R,S$ be Dedekind domain with torsion class group, and let $h:G(R)\longrightarrow G(S)$ be a homeomorphism of Golomb topologies. Then the restriction of $h$ to $R\setminus P$ is a homeomorphism between $R\setminus P$ with the $P$-topology and $S\setminus h(P)$ with the $h(P)$-topology.
\end{teor}
\begin{proof}
If $\Omega\subseteq R\setminus P$ is clopen in $R\setminus P$, then $h(\Omega)$ is clopen in $S\setminus h(P)$. Hence, the basic open sets of the $P$-topology go to open sets in the $h(P)$-topology; since the same holds for $h^{-1}$, the restriction of $h$ is a homeomorphism between the $P$-topology and the $h(P)$-topology.
\end{proof}

We end this section by determining the closure of a subset in the $P$-topology.
\begin{prop}\label{prop:chiusPtop}
Let $Y\subseteq R\setminus P$, and let $X$ be the closure of $Y$ in the $P$-topology. Then,
\begin{equation*}
X=\bigcap_{n\geq 1}\pi_n^{-1}(\pi_n(Y)).
\end{equation*}
\end{prop}
\begin{proof}
Let $g$ be in the intersection: then, for every $n$, there is $a_n\in Y$ such that $\pi_n(g)=\pi_n(a_n)$, that is, $g-a_n\in P^n$. Hence, $g\in X$. Conversely, if $g$ is in the closure then for every $n$ there is $a_n\in Y$ such that $g-a_n\in P^n$; that is, $\pi_n(g)\in\pi_n(Y)$, as claimed.
\end{proof}

\section{The groups $H_n(P)$}\label{sect:HnP}
Let $R,S$ be Dedekind domain with torsion class group, and let $P$ be a prime ideal of $R$. Let $h:G(R)\longrightarrow G(S)$ be a homeomorphism. By Theorem \ref{teor:radical}, $h(P)$ is a prime ideal of $S$. A natural question is whether this result can be generalized to cosets: that is, if $a\in R\setminus P$, does $h(a+P)=h(a)+h(P)$? In particular, if $h(1)=1$, does $h(1+P)=1+h(P)$? We are not able to prove this result; therefore, our strategy will be to use Proposition \ref{prop:GDelta}, the $P$-topology and the group structure of $U(R/P^n)$ to obtain ``approximate'' results. We collect in this section some technical lemmas which will be useful in the following sections.

The basic idea is to consider the quotients of $R$ into $R/P^n$, or rather the unit groups $U(R/P^n)$. However, as the multiplication by a unit of $R$ is a self-homeomorphism of $G(R)$, it is more useful to study the groups
\begin{equation*}
H_n(P):=U(R/P^n)/\pi_n(U(R)),
\end{equation*}
where $\pi_n:R\longrightarrow R/P^n$ is the canonical surjective map. Furthermore, we denote by $\theta_n:U(R/P^n)\longrightarrow H_n(P)$ the canonical quotient, and by $\widetilde{\pi}_n=\theta_n\circ\pi_n:R\setminus P\longrightarrow H_n(P)$ the composition of the previous maps.

For different $n$, these maps are linked in the following way.
\begin{lemma}\label{lemma:mapHn}
For every $n\geq 1$, there is a surjective map $\lambda_n:H_{n+1}(P)\longrightarrow H_n(P)$ such that the following diagram commutes:
\begin{equation*}
\begin{tikzcd}
R\setminus P\arrow[two heads]{r}{\pi_{n+1}}\arrow[equal]{d} & U(R/P^{n+1})\arrow[two heads]{r}{\theta_{n+1}}\arrow[two heads]{d} & H_{n+1}(P)\arrow[two heads]{d}{\lambda_n}\\
R\setminus P\arrow[two heads]{r}{\pi_n} & U(R/P^n)\arrow[two heads]{r}{\theta_n} & H_n(P).
\end{tikzcd}
\end{equation*}
\end{lemma}
\begin{proof}
The natural map $R/P^{n+1}\longrightarrow R/P^n$ restricts to a surjective map $\lambda':U(R/P^{n+1})\longrightarrow U(R/P^n)$. Furthermore, $\lambda'(\pi_{n+1}(U(R)))=\pi_n(U(R))$; hence, $\lambda'$ induces $\lambda_n$, which remains surjective.
\end{proof}

In particular, if $L$ is a subgroup of $H_n(L)$, then we have a sequence of surjective maps
\begin{equation}\label{eq:telescope}
\begin{tikzcd}
L & \arrow{l}[swap]{\lambda_n}L_1:=\lambda_n^{-1}(L) & \arrow{l}[swap]{\lambda_{n+1}^{-1}} L_2:=\lambda_{n+1}(L_1) & \arrow{l}[swap]{\lambda_{n+2}}\cdots
\end{tikzcd}
\end{equation}
where each $L_i$ is a subgroup of $H_{n+i}(L)$. Furthermore, for every $i$, the index $[H_{n+i}(P):L_i]$ is equal to the index $[H_n(P):L]$ of $L$, and in particular it is constant. We call the sequence $\{L,L_1,\ldots,\}$ the \emph{telescopic sequence} of $L$. 

When $L=H_1(P)$, the telescopic sequence of $L$ is just the sequence $\{H_1(P),H_2(P),\ldots\}$. We distinguish two classes of behavior.

One case is when the maps $\lambda_n$ are isomorphisms for every $n\geq N$: in this case, all the information about the $H_n(P)$ ``stops at $N$''. If $R/P$ is finite (and thus also $U(R/P^n)$ and $H_n(P)$ are finite for every $n\geq 1$) then in particular the sequence of the cardinalities of the $H_n(P)$ is bounded. This happens, for example, if $R=\insZ[1/p]$ for some prime number $p$.

The second case is when there are infinitely many $\lambda_n$ that are not isomorphisms, as it happens for $R=\insZ$: in this case, the structure of the $H_n(P)$ depends on all the groups. If $R/P$ is finite, this implies that the sequence of the cardinalities of the $H_n(P)$ is not bounded; however, a part of the structure of these groups still remain fixed, as we show next. Given an abelian group $L$ and a prime number $p$, the \emph{non-$p$-component} of $L$ is the subgroup of $L$ formed by the elements whose order is coprime with $p$.
\begin{lemma}\label{lemma:etaP}
Let $R$ be a Dedekind domain and let $P$ be a prime ideal such that $R/P$ is finite; let $p$ be the characteristic of $R/P$. Then, there is an integer $\eta(P)$, coprime with $p$, such that, for all $n\geq N$, the non-$p$-component of $H_n(P)$ has order $\eta(P)$.
\end{lemma}
\begin{proof}
Let $H'_k(P)$ be non-$p$-component of $H_k(P)$, and let $\eta_k(P)$ be its cardinality. Since $\lambda_k(H'_{k+1}(P))=H'_k(P)$, $\eta_k(P)$ divides $\eta_{k+1}(P)$, and thus $\{\eta_k(P)\}_{k\inN}$ is an ascending chain with respect to the divisibility order. Furthermore, if $|R/P|=p^e$, then $|U(R/P^n)|=p^{e(n-1)}(p^e-1)$, and thus $\eta_k(P)$ divides $p^e-1$; hence, the chain is bounded above and thus finite. It follows that it stabilizes at some value $\eta(P)$.
\end{proof}

Several results in the following sections will be valid only under the assumption that the $H_n(P)$ are cyclic. This forces a rather severe limit on the cardinalities of the residue fields.
\begin{lemma}\label{lemma:almcyc-cardRP}
Let $R$ be a Dedekind domain, and let $P$ be a prime ideal of $R$. If $U(R)$ is discrete in the $P$-topology, and $H_n(P)$ is cyclic for every $n$, then $|R/P|$ is a prime number.
\end{lemma}
\begin{proof}
Since $U(R)$ is discrete in the $P$-topology, there is an $N\geq 2$ such that $1+P^{N-1}$ contains no units different from $1$. Let $p\in P^{N-1}\setminus P^N$, and for every $a\in R$ let $\sigma(a):=\widetilde{\pi}_N(1+ap)$. Then, $p^2\in P^N$ and thus
\begin{equation*}
\sigma(a)\sigma(b)=\widetilde{\pi}_N((1+ap)(1+bp))=\widetilde{\pi}_N(1+(a+b)p)=\sigma(a+b).
\end{equation*}
Hence, the map $a\mapsto\sigma(a)$ is a homomorphism from the additive group of $R$ to the multiplicative subgroup $\Sigma:=\{\sigma(a)\mid a\in R\}$ of $H_N(P)$. Furthermore,
\begin{equation*}
\begin{aligned}
\ker\sigma= & \{a\in R\mid \widetilde{\pi}_N(1+ap)=\widetilde{\pi}_N(1)\}=\\
= & \{a\in R\mid 1+ap\in U(R)+P^N\}=P
\end{aligned}
\end{equation*}
by the choice of $N$ and $p$. Therefore, $\sigma$ factors into an embedding of $(R/P,+)$ inside $H_N(P)$; since $H_N(P)$ is cyclic, it follows that also $(R/P,+)$ is cyclic, and since $R/P$ is a field it must be equal to $\ins{F}_p$ for some prime number $p$. In particular, $|R/P|$ is prime.
\end{proof}

Note that the fact that $|R/P|$ is a prime number does not guarantee that $H_n(P)$ is cyclic: for example, if $R=\ins{F}_p[X]$, where $p>2$ is a prime number, and $P=(X)$, then $H_3(P)$ has $p^2$ elements, but every element has order $p$.

\section{Closure of powers}\label{sect:clos}

In isolation, the $P$-topology is not very interesting: indeed, since it coincides with the $P$-adic topology, it makes $R\setminus P$ into a metric space with no isolated points. In particular, if $R$ is countable then $R\setminus P$ is homeomorphic to $\insQ$ \cite{sierpinski-Q,dagupta-atlas}, and thus a homeomorphism between the $P$-topology of $R\setminus P$ and the $Q$-topology of $S\setminus Q$ does not give much information. However, by Proposition \ref{prop:GDelta}, a homeomorphism $h$ between Golomb spaces carries a lot more structure.

In the following, we shall mostly restrict ourselves to Dedekind domains with torsion class group. Given $a\in R\setminus P$, set
\begin{equation*}
\pow{a}:=\{ua^t\mid u\in U(R),t\inN^+\}.
\end{equation*}
We want to study the closure of $\pow{a}$ in the $P$-topology.

\begin{prop}\label{prop:chiuspow}
Let $R$ be a Dedekind domain, $P$ a prime ideal, $a\in R\setminus P$; let $X$ be the closure of $\pow{a}$ in the $P$-topology. Then, the following hold.
\begin{enumerate}[(a)]
\item\label{prop:chiuspow:princtors-pi} If $\pi_n(a)$ is torsion in $U(R/P^n)$ for every $n\geq 1$ then
\begin{equation*}
X=\bigcap_{n\geq 1}\pi_n^{-1}\bigl{(}\langle\pi_n(a),\pi_n(U(R))\rangle\bigl{)}.
\end{equation*}
\item\label{prop:chiuspow:princtors-Hn} If $\widetilde{\pi}_n(a)$ is torsion in $H_n(P)$ for every $n\geq 1$ then
\begin{equation*}
X=\bigcap_{n\geq 1}\widetilde{\pi}_n^{-1}(\langle \widetilde{\pi}_n(a)\rangle).
\end{equation*}
\end{enumerate}
\end{prop}
\begin{proof}
\ref{prop:chiuspow:princtors-pi} If $\pi_n(a)$ is torsion with order $k$, then
\begin{equation*}
\begin{aligned}
\pi_n(\pow{a})= & \{\pi_n(u)\pi_n(a)^t+P^n\mid u\in U(R),t\inN^+\}=\\
= & \{\pi_n(u)\pi_n(a)^t+P^n\mid u\in U(R),t\in\{1,\ldots,k\}\}
\end{aligned}
\end{equation*}
is exactly the subgroup generated by $\pi_n(a)$ and $\pi_n(U(R))$. The claim now follows from Proposition \ref{prop:chiusPtop}.

\ref{prop:chiuspow:princtors-Hn} follows as the previous point, noting that $\widetilde{\pi}_n$ sends all of $U(R)$ into the identity.
\end{proof}

In general, we would like for the sequence $\widetilde{\pi}_n^{-1}(\langle \widetilde{\pi}_n(a)\rangle)$ to stabilize: in this case, we could study the closures of the $\pow{a}$ simply by studying the subgroups of the $H_n(P)$. In general, this is not true: for example, this happens if $a=1$ (so $\pow{a}=U(R)$) and the cardinality of $H_n(P)$ goes to infinity. However, we can characterize this case; we distinguish the two behaviors of the $\lambda_n$.

\begin{prop}\label{prop:chius-powa-sgr-bounded}
Let $R$ be a Dedekind domain, $P$ a prime ideal, $X\subseteq R\setminus P$. Suppose that the canonical surjections $\lambda_n:H_{n+1}(P)\longrightarrow H_n(P)$ are isomorphisms for $n\geq N$. Then, the following are equivalent.
\begin{enumerate}[(i)]
\item $X$ is the closure of $\pow{a}$ for some $a\in R\setminus P$;
\item $X=\widetilde{\pi}_N^{-1}(L)$ for some cyclic subgroup $L$ of $H_N(P)$.
\end{enumerate}
\end{prop}
\begin{proof}
For every $a\in R\setminus P$, we have $\lambda_{N+k}(\langle\widetilde{\pi}_{N+k+1}(a)\rangle)=\langle\widetilde{\pi}_{N+k}(a)\rangle$ for every $k\geq 0$. Hence, $\widetilde{\pi}_N^{-1}(\langle \widetilde{\pi}_N(a)\rangle)=\widetilde{\pi}_{N+k}^{-1}(\langle \widetilde{\pi}_{N+k}(a)\rangle)$ for every $k\geq 0$. The claim follows.
\end{proof}

When the canonical surjections are not isomorphisms, the picture is more complicated. For simplicity, we restrict to the case where $R/P$ is finite.
\begin{prop}\label{prop:chius-powa-sgr}
Let $R$ be a Dedekind domain, $P$ a prime ideal, $a\in R\setminus P$; let $X$ be the closure of $\pow{a}$ in the $P$-topology. Suppose that $R/P$ is finite and that there are infinitely many $n$ such that $\lambda_n:H_{n+1}(P)\longrightarrow H_n(P)$ is not an isomorphism. Then, the following are equivalent:
\begin{enumerate}[(i)]
\item\label{prop:chius-powa-sgr:stab} the chain $\{\widetilde{\pi}_n^{-1}(\langle \widetilde{\pi}_n(a)\rangle)\}_{n\inN}$ stabilizes;
\item\label{prop:chius-powa-sgr:sgr} $X=\widetilde{\pi}_N^{-1}(L)$ for some subgroup $L$ of $H_N(P)$;
\item\label{prop:chius-powa-sgr:tseq1} there is an $N$ such that every element of the telescopic sequence of $\langle\widetilde{\pi}_N(a)\rangle$ is generated by the image of $a$;
\item\label{prop:chius-powa-sgr:tseq2} there is an $N$ such that every element of the telescopic sequence of $\langle\widetilde{\pi}_N(a)\rangle$ is cyclic, and the order of $\widetilde{\pi}_n(a)$ goes to infinity as $n\to\infty$.
\end{enumerate}
\end{prop}
\begin{proof}
\ref{prop:chius-powa-sgr:stab} $\Longrightarrow$ \ref{prop:chius-powa-sgr:sgr} If the chain stabilizes at $N$, that is, if $\widetilde{\pi}_N^{-1}(\langle \widetilde{\pi}_N(a)\rangle)=\widetilde{\pi}_{N+k}^{-1}(\langle \widetilde{\pi}_{N+k}(a)\rangle)$ for all $k\geq 0$, then $X=\widetilde{\pi}_N^{-1}(L)$ with $L:=\langle \widetilde{\pi}_N(a)\rangle$.

\ref{prop:chius-powa-sgr:sgr} $\Longrightarrow$ \ref{prop:chius-powa-sgr:tseq1} If $X=\widetilde{\pi}_N^{-1}(L)$, then $\widetilde{\pi}_N(X)=L$, and thus by Proposition \ref{prop:chiuspow}\ref{prop:chiuspow:princtors-pi} $L=\langle \widetilde{\pi}_N(a)\rangle$. We have a commutative diagram
\begin{equation*}
\begin{tikzcd}
X\arrow[equal]{d}\arrow{r}{\widetilde{\pi}_N} & L\\
X\arrow{r}{\widetilde{\pi}_{N+1}} & \lambda_N^{-1}(L)\arrow{u}{\lambda_N};
\end{tikzcd}
\end{equation*}
however, we also have $\widetilde{\pi}_{N+1}(X)=\langle\widetilde{\pi}_{N+1}(a)\rangle$, and thus the telescopic sequence of $L$ is formed by the subgroups is generated by (the image of) $a$ in the various $H_{N+k}(P)$.

\ref{prop:chius-powa-sgr:tseq1} $\Longrightarrow$ \ref{prop:chius-powa-sgr:tseq2} Since there are infinitely many $n$ such that $\lambda_n$ is not an isomorphism, the cardinality of $H_n(P)$ goes to infinity; since the index remains fixed among the elements of a telescopic sequence, it follows that the cardinality of the $\langle\widetilde{\pi}_n(a)\rangle$ is unbounded, as claimed.

\ref{prop:chius-powa-sgr:tseq2} $\Longrightarrow$ \ref{prop:chius-powa-sgr:stab} Let $\sigma_n$ be the order of $\widetilde{\pi}_n(a)$.

By Lemma \ref{lemma:etaP}, $\sigma_n=p^{k(n)}d_n$ for some $k(n)\geq 0$ and some $d_n$ dividing $\eta(P)$; since the sequence $\{\sigma_n\}_{n\inN}$ is unbounded, we can find $N'$ such that $d_{N'}=d_{N'+k}$ for every $k\geq 0$. Furthermore, by the hypothesis, we can find $N\geq N'$ such that every element of the telescopic sequence $\{L,L_1,\ldots,\}$ of $L:=\langle\widetilde{\pi}_N(a)\rangle$ is cyclic. We claim that each $L_i$ is generated by the image of $a$.

Indeed, by construction we have $|L_k|=p^{s(k)}|L|=p^{s(k)}\sigma_N$ for every $k\geq 0$ (and some nonnegative function $k\mapsto s(k)$). If $\phi$ be the Euler totient, the number of generators of $L_k$ is
\begin{equation*}
\phi(|L_k|)=\phi(\sigma_{N+k})=\phi(p^{s(k)}\sigma_N)=p^{s(k)}\phi(\sigma_N),
\end{equation*}
since $p|\sigma_N$. Hence, every generator of $L$ lifts to a generator of $L_k$; therefore, $\widetilde{\pi}_N^{-1}(\langle \widetilde{\pi}_N(a)\rangle)=\widetilde{\pi}_{N+k}^{-1}(\langle \widetilde{\pi}_{N+k}(a)\rangle)$ for all $k\geq 0$, as claimed.
\end{proof}

One problem in applying the previous proposition to the Golomb topology is that we don't know if the sets $\pow{a}$ are invariant with respect to homeomorphism of Golomb spaces. However, if $R,S$ are principal ideal domains, and $q\in R$ is a prime element (i.e., if $qR$ is a prime ideal) then $\pow{q}=G_{\{qR\}}(R)$, and thus by Proposition \ref{prop:GDelta} a homeomorphism $h:G(R)\longrightarrow G(S)$ carries $\pow{q}$ to $\pow{q'}$, for some prime element $q'$ of $S$. Therefore, it carries the closure of $\pow{q}$ in the $P$-topology to the closure of $\pow{q'}$ in the $h(P)$-topology.

More generally, suppose $R$ is a Dedekind domain with torsion class group. If $Q^t=qR$ is the smallest power of $Q$ that is a principal ideal, we say that $q$ is an \emph{almost prime} element; equivalently, an almost prime element is an irreducible element generating a primary ideal. In this case, we still have $\pow{q}=G_{Q}(R)$, since if $xR$ is a $Q$-primary ideal then $xR$ must be in the form $(Q^t)^k$ for some $k$. In particular, we must still have $h(\pow{q})=\pow{q'}$ for some almost prime element $q'$ of $S$. More precisely, the unique prime ideal containing $q'$ will be the image $h(Q)$ of $Q$, the only prime ideal containing $q$.

\begin{defin}
Let $P$ be a prime ideal. We define $\setclos{P}$ as the set of closures of $\pow{q}$, as $q$ ranges among the almost prime elements of $R$ outside $P$.
\end{defin}

The previous discussion shows the following.
\begin{prop}\label{prop:omef-XP}
Let $R,S$ be two Dedekind domains with torsion class group, and let $h:G(R)\longrightarrow G(S)$ be a homeomorphism. Then, the map
\begin{equation*}
\begin{aligned}
\overline{h}\colon\setclos{P} & \longrightarrow\setclos{h(P)},\\
X & \longmapsto h(X)
\end{aligned}
\end{equation*}
is an order isomorphism (when $\setclos{P}$ and $\setclos{h(P)}$ are endowed with the containment order).
\end{prop}

We are now interested in studying the order structure of $\setclos{P}$; since we also need to have plenty of almost prime elements, we introduce the following definition.
\begin{defin}
Let $R$ be a principal ideal domain and $P$ a prime ideal of $R$. We say that $R$ is \emph{Dirichlet at $P$} if, for every $a\in R\setminus P$ and every $n\geq 1$ the coset $a+P^n$ contains at least one almost prime element.
\end{defin}
For example, by Dirichlet's theorem on primes in arithmetic progressions (see e.g. \cite[Chapter 4]{davenport-analytic} or \cite[Chapter 7]{apostol}), $\insZ$ is Dirichlet at each of its primes. An equivalent condition is that the set of almost prime elements of $R$ is dense in $R\setminus P$ under the $P$-topology. Note that it is not known if a homeomorphism of Golomb spaces sends almost prime elements in almost prime elements, and thus this condition may not be a topological invariant.


We shall use the following terminology.
\begin{defin}
Let $R$ be a Dedekind domain, and let $P$ be a prime ideal of $R$. We say that $P$ is \emph{almost cyclic} if $R/P$ is finite and $H_n(P)$ is cyclic for every $n\geq 1$.
\end{defin}

Our next step is to link $\setclos{P}$ with the subgroups of the $H_n(P)$. We first show how to compare subgroups living in different $H_n(P)$.
\begin{lemma}\label{lemma:ug-subg}
Let $R$ be a Dedekind domain, and let $P$ be an almost cyclic prime ideal. Let $L$ and $L'$ be, respectively, subgroups of $H_n(P)$ and $H_m(P)$. Then, $\widetilde{\pi}_n^{-1}(L)\subseteq \widetilde{\pi}_m^{-1}(L')$ if and only if $[H_m(P):L']$ divides $[H_n(P):L]$; in particular, $\widetilde{\pi}_n^{-1}(L)=\widetilde{\pi}_m^{-1}(L')$ if and only if $[H_m(P):L']=[H_n(P):L]$.
\end{lemma}
\begin{proof}
Without loss of generality, suppose $m\geq n$, and let $\lambda:H_n(P)\longrightarrow H_m(P)$ be the canonical surjective map. Then, $\lambda(L)$ is a subgroup of $H_m(P)$ of index $[H_n(P):L]$; since $H_m(P)$ is cyclic, we have $\lambda(L)\subseteq L'$ if and only if $[H_n(P):L]$ is a multiple of $[H_m(P):L']$, as claimed.

The ``in particular'' part follows immediately.
\end{proof}

\begin{prop}\label{prop:caratt-XP}
Let $R$ be a Dedekind domain with torsion class group, and let $P$ be an almost cyclic prime ideal. Then, the following hold.
\begin{enumerate}[(a)]
\item\label{prop:caratt-XP:->sgr} Let $X$ be the closure of $\pow{q}$ in the $P$-topology. If $\pow{q}$ is disjoint from the closure of $U(R)$ (with respect to the $P$-topology), then there is an $n\geq 1$ and a subgroup $H$ of $H_n(P)$ such that $X=\widetilde{\pi}_n^{-1}(H)$.
\item\label{prop:caratt-XP:sgr->} If $R$ is Dirichlet at $P$, then $\widetilde{\pi}_n^{-1}(H)\in\mathcal{X}(P)$ for every subgroup $H$ of $H_n(P)$.
\end{enumerate}
\end{prop}
\begin{proof}
Let $p$ be the characteristic of $R/P$.

\ref{prop:caratt-XP:->sgr} If the cardinality of the $H_n(P)$ is bounded, the claim follows from Proposition \ref{prop:chius-powa-sgr-bounded}.

If the cardinality is unbounded, let $q$ be an almost prime element such that $X$ is the closure of $\pow{q}$, and let $\sigma_n$ be the order of $\widetilde{\pi}_n(q)$. Suppose that $\{\sigma_n\}_{n\inN}$ is bounded, and let $\sigma$ be its maximum; then, $\widetilde{\pi}_n(q)^\sigma$ is the identity in $H_n(P)$ for all $n$, i.e., $\pi_n(q)^\sigma\in U(R)+P^n$ for every $n$. However, this implies that $q^\sigma$ is in the closure of $U(R)$ in the $P$-topology, a contradiction. Therefore, $\sigma_n$ becomes arbitrary large and the claim follows from Proposition \ref{prop:chius-powa-sgr}.

\ref{prop:caratt-XP:sgr->} If the cardinality of the $H_n(P)$ is bounded, then there is an $a\in R\setminus P$ and an $N$ such that $X:=\widetilde{\pi}_N^{-1}(H)$ is the closure of $\pow{a}$ in the $P$-topology; since $R$ is Dirichlet at $P$ there is an almost prime element $q\in a+P^N$, and $X$ is the closure of $\pow{q}$ in the $P$-topology, as claimed.

Suppose that the cardinality of the $H_n(P)$ is not bounded. Let $N\geq n$ be big enough such that the non-$p$-component of $H_N(P)$ has cardinality $\eta(P)$, and choose $k>N$ such that $H_k(P)>H_N(P)$. Let $L$ be the element of the telescopic sequence of $H$ that is contained in $H_k(P)$. Then, $L$ is cyclic, and thus there is an $a\in R\setminus P$ such that $\widetilde{\pi}_k(a)$ generates $L$; as in the proof of Proposition \ref{prop:chius-powa-sgr}, the fact that $p$ divides the cardinality of $L$ implies that every element of the telescopic sequence of $L$ is generated by the image of $a$. Since $R$ is Dirichlet at $P$, we can find an almost prime element $q\in a+P^k$; then, $X$ is the closure of $\pow{q}$, and in particular $X\in\setclos{P}$, as claimed.
\end{proof}

\begin{cor}\label{cor:RmenP-almcyc}
Let $R$ be a Dedekind domain with torsion class group, and let $P$ be a prime ideal of $R$. Suppose that $U(R)$ is closed in the $P$-topology. Then, the following hold.
\begin{enumerate}[(a)]
\item If $R\setminus P\in\setclos{P}$, then $P$ is almost cyclic.
\item If $R$ is Dirichlet at $P$ and $P$ is almost cyclic, then $R\setminus P\in\setclos{P}$.
\end{enumerate}
\end{cor}
\begin{proof}
If $R\setminus P\in\setclos{P}$, then there is an almost prime element $q$ such that $R\setminus P$ is the closure of $\pow{q}$. By Proposition \ref{prop:chiuspow}, each $H_n(P)$ is generated by the image of $q$, and in particular they are all cyclic.

Conversely, suppose $P$ is almost cyclic. By Proposition \ref{prop:caratt-XP}\ref{prop:caratt-XP:sgr->}, $\widetilde{\pi}_n^{-1}(H)\in\setclos{P}$ for every subgroup of the $H_n(P)$; in particular, this holds for $H=H_n(P)$, for which we have $\widetilde{\pi}_n^{-1}(H)=R\setminus P$.
\end{proof}

\begin{cor}\label{cor:almcyc-topinv}
Let $R,R'$ be Dedekind domains with torsion class group and let $P$ be a prime ideal of $R$; suppose that $U(R)$ is closed in the $P$-topology. Let $h:G(R)\longrightarrow G(R')$ be a homeomorphism and let $P':=h(P)$.
\begin{enumerate}[(a)]
\item If $R$ is Dirichlet at $P$ and $P$ is almost cyclic then $P'$ is almost cyclic.
\item If also $R'$ is Dirichlet at $P'$, then $P$ is almost cyclic if and only if $P'$ is almost cyclic.
\end{enumerate}
\end{cor}
\begin{proof}
If $R$ is Dirichlet at $P$ and $P$ is almost cyclic, then by Corollary \ref{cor:RmenP-almcyc} $R\setminus P\in\setclos{P}$; hence, $R'\setminus P'=h(R\setminus P)\in h(\setclos{P})=\setclos{P'}$. Applying again the corollary we see that $P'$ is almost cyclic.

The second part follows by considering the inverse $h^{-1}:G(R')\longrightarrow G(R)$.
\end{proof}

Set now
\begin{equation*}
\setdiv{P}:=\{d\inN\mid d\text{~divides~}H_n(P)\text{~for some~}n\};
\end{equation*}
then, $\setdiv{P}$ has a natural order structure given by the divisibility relation (i.e., $a\leq b$ if and only if $a|b$). From a structural point of view, the previous proposition implies the following result.

\begin{figure}
\caption{The structure of $\setdiv{p\insZ}$ for $p=41$. In this case, $\eta(p\insZ)=20=2^2\cdot 5$.}
\begin{equation*}
\begin{tikzcd}[every arrow/.append style={dash},column sep=tiny]
& & & & & & & & & 20p^2\\
& & & & & 20p\arpp & & & 4p^2\arrow{ur} & 10p^2\arrow{u}\\
& 20\arpp & & & 4p\arrow{ur}\arpp & 10p\arrow{u}\arpp & & & 2p^2\arrow{u}\arrow{ur} & & 5p^2\arrow{ul} & & & \cdots\\
4\arrow{ur}\arpp & 10\arrow{u}\arpp & & & 2p\arrow{u}\arrow{ur}\arpp & & 5p\arrow{ul}\arpp & & & p^2\arrow{ur}\arrow{ul}\arrow{urrrr}\\
2\arrow{u}\arrow{ur}\arpp & & 5\arrow{ul}\arpp & & & p\arrow{ur}\arrow{ul}\arrow{urrrr}\\
& & & 1\arrow{ulll}\arrow{ul}\arrow{urr}
\end{tikzcd}
\end{equation*}
\end{figure}

\begin{teor}\label{teor:corresp-Theta}
Let $R$ be a Dedekind domain with torsion class group, $P$ an almost cyclic prime ideals, and suppose that $U(R)$ is closed in the $P$-topology. Let $\Theta$ be the map
\begin{equation*}
\begin{aligned}
\Theta\colon\setclos{P} & \longrightarrow\setdiv{P},\\
X=\widetilde{\pi}_n^{-1}(L) & \longmapsto[H_n(P):H].
\end{aligned}
\end{equation*}
Then, the following hold.
\begin{enumerate}[(a)]
\item $\Theta$ is well-defined, injective and order-reversing.
\item If $R$ is Dirichlet at $P$, then $\Theta$ is surjective, and thus $\Theta$ is an order-reversing isomorphism.
\end{enumerate}
\end{teor}
\begin{proof}
Since $U(R)$ is closed in the $P$-topology, and every $\pow{q}$ is disjoint from $U(R)$, by Proposition \ref{prop:caratt-XP}\ref{prop:caratt-XP:->sgr} every $X\in\setclos{P}$ is in the form $\widetilde{\pi}_n^{-1}(L)$; by Lemma \ref{lemma:ug-subg}, if it is also equal to $\widetilde{\pi}_{n'}^{-1}(L')$ then the index of $L$ and $L'$ are the same, and thus $\Theta$ is well-defined. The same Lemma \ref{lemma:ug-subg} implies also that $\Theta$ is injective and order-reversing.

If $R$ is Dirichlet at $P$, we can apply Proposition \ref{prop:caratt-XP}\ref{prop:caratt-XP:sgr->}, and thus $\Theta$ is also surjective. It follows that $\Theta$ is an order-reversing isomorphism.
\end{proof}

The previous theorem implies that, under good hypothesis, the structure of $\setdiv{P}$ is a topological invariant of the Golomb topology; in particular, if $h:G(R)\longrightarrow G(S)$ is a homeomorphism, then Proposition \ref{prop:omef-XP} can be extended to a chain of bijections
\begin{equation}\label{eq:chain}
\setdiv{P}\xrightarrow{\Theta_P^{-1}}\mathcal{X}(P)\xrightarrow{~~\overline{h}~~}\mathcal{X}(h(P))\xrightarrow{\Theta_{h(P)}}\setdiv{h(P)}
\end{equation}
whose composition gives an order isomorphism between $\setdiv{P}$ and $\setdiv{h(P)}$.

We shall use the following shorthand.
\begin{defin}
Let $z,z'\inN$, and let $z=p_1^{e_1}\cdots p_k^{e_k}$ and $z'=q_1^{f_1}\cdots q_r^{f_r}$ be their factorizations. We say that $z$ and $z'$ \emph{have the same factorization structure} if $k=r$ and, after a permutation, $e_i=f_i$ for every $i$.
\end{defin}

\begin{prop}\label{prop:DP}
Let $R,R'$ be two Dedekind domain with torsion class group, and suppose there is a homeomorphism $h:G(R)\longrightarrow G(R')$. Let $P$ be an almost cyclic prime ideal of $R$, and let $P':=h(P)$; suppose that $R'/P'$ is finite, that $U(R)$ is closed in the $P$-topology, that $R$ is Dirichlet at $P$ and that $R'$ is Dirichlet at $P'$. Then, the following hold.
\begin{enumerate}[(a)]
\item The sequence $\{|H_n(P)|\}_{n\inN}$ is bounded if and only if $\{|H_n(P')|\}_{n\inN}$ is bounded.
\item If $|H_n(P)|=z$ and $|H_n(P)|=z'$ for all $n\geq N$, then $z$ and $z'$ have the same factorization structure.
\item If $\{|H_n(P)|\}_{n\inN}$ and $\{|H_n(P')|\}_{n\inN}$ are unbounded, then $\eta(P)$ and $\eta(P')$ have the same factorization structure.
\end{enumerate}
\end{prop}
\begin{proof}
Since $h$ is a homeomorphism in the $P$-topology, $U(R')=h(U(R))$ is closed in the $P'$-topology; furthermore, by Corollary \ref{cor:almcyc-topinv}, $P'$ is almost cyclic. By Proposition \ref{prop:omef-XP}, there is an order isomorphism between $\setdiv{P}$ and $\setdiv{P'}$.

The sequence $\{|H_n(P)|\}_{n\inN}$ is bounded if and only if it is finite, which happens if and only if $\setdiv{P}$ is finite. Hence, $\{|H_n(P)|\}_{n\inN}$ is bounded if and only if $\{|H_n(P')|\}_{n\inN}$ is bounded.

If $|H_n(P)|=z$ for all large $n$, then $\setdiv{P}$ is just the set of divisors of $z$; in particular, the minimal elements of $\setdiv{P}\setminus\{1\}$ correspond to the distinct prime factors of $z$. Since the same happens for $\setdiv{P'}$, the number of distinct prime factors of $z$ and $z'$ is the same. Furthermore, the exponent of $p$ in $z$ is equal to the number of elements of $\setdiv{P}$ that are divisible only by $p$; hence, it depends only on the structure of $\setdiv{P}$, and thus it doesn't change passing from $\setdiv{P}$ to $\setdiv{P'}$.

In the same way, if $\{|H_n(P)|\}_{n\inN}$ is unbounded, the minimal elements of $\setdiv{P}$ correspond to $p$ (the cardinality of $R/P$) and the prime factors of $\eta(P)$; furthermore, $p$ is the unique minimal element such that there are infinitely many $x$ that are multiple of $p$ but of no other prime factor. Hence, in the chain of bijections \eqref{eq:chain} $p$ gets sent to $p'$, the cardinality of $R'/P'$. The divisors of $\eta(P)$ are the elements of $\setdiv{P}$ that are not divisible by $p$; hence, the divisors of $\eta(P)$ correspond to the divisors of $\eta(P')$. As in the previous case, this implies that $\eta(P)$ and $\eta(P')$ have the same factorization structure.
\end{proof}

\section{The correspondence at powers of $p$}\label{sect:Yk}

Proposition \ref{prop:DP} gives a very strong restrictions for the image of a prime ideal under a homeomorphism of Golomb spaces. For example, suppose $R=\insZ$. Then, every prime ideal is almost cyclic, and it is easy to see that
\begin{equation*}
\eta(p\insZ)=\begin{cases}
1 & \text{if~}p=2\\
\frac{p-1}{2} & \text{if~}p>2.
\end{cases}
\end{equation*}
The only prime ideals $p\insZ$ such that $\eta(p\insZ)=1$ (and so $\eta(p\insZ)$ has an empty factorization) are $2$ and $3$; it follows that, for every self-homeomorphism $h$ of $G(\insZ)$, $h(2\insZ)$ can be equal only to $2\insZ$ or $3\insZ$. Likewise, $\eta(5\insZ)=2$ is prime, and thus $h(5\insZ)$ must be equal to $(2q+1)\insZ$ for some prime number $q$ such that $2q+1$ is prime. 

In this section, we use a finer analysis of the structure of $\setdiv{P}$ to obtain even more. We concentrate on sets in the form
\begin{equation*}
Y_k(P):=\Theta^{-1}(\eta(P)p^k)
\end{equation*}
where $\Theta$ is the map of Theorem \ref{teor:corresp-Theta}.

\begin{prop}\label{prop:Yk}
Preserve the hypothesis and the notation of Proposition \ref{prop:DP}, and suppose that $\{|H_n(P)|\}_{n\inN}$ is unbounded; let $p:=|R/P|$ and $p':=|R'/P'|$ . Then, the following hold.
\begin{enumerate}[(a)]
\item Let $h^\star:=\Theta_{P'}\circ h\circ\Theta_P$. Then, $h^\star(\eta(P)p^k)=\eta(P')(p')^k$ for every $k\geq 0$.
\item $h(Y_k(P))=Y_k(P')$.
\end{enumerate}
\end{prop}
\begin{proof}
As we saw in the proof of Proposition \ref{prop:DP}, the maximal elements of $\setdiv{P}\setminus\{1\}$ correspond to $p$ and the prime factors of $\eta(P)$; moreover, $p$ is the unique minimal element of $\setdiv{P}\setminus\{1\}$ with infinitely many multiples that are not divisible by any other prime. Hence, $h^\star(p)=p'$. Furthermore, $\eta(P)$ is the largest element of $\setdiv{P}$ that is not a multiple in $p$, and thus $h^\star(\eta(P))$ is the largest element of $\setdiv{P'}$ that is not a multiple of $h^\star(p)=p'$; that is, $h^\star(\eta(P))=\eta(P')$.

Consider now the multiples of $\eta(P)$ in $\setdiv{P}$: they are all in the form $\eta(P)p^k$ for some $k\geq 0$. The map $h^\star$ restricts to an order isomorphism between the multiples of $\eta(P)$ and the multiples of $\eta(P')$; hence, it must be $h^\star(\eta(P)p^k)=\eta(P')(p')^k$, as claimed.

By turning \eqref{eq:chain} inside-out and using the previous part of the proof, we see that
\begin{equation*}
\begin{aligned}
h(Y_k(P))= & (\Theta_{P'}^{-1}\circ h^\star\circ\Theta_P)(Y_k(P))=\\
= & (\Theta_{P'}^{-1}\circ h^\star)(\eta(P)p^k)= \Theta_{P'}^{-1}(\eta(P')(p')^k)=Y_k(P').
\end{aligned}
\end{equation*}
The claim is proved.
\end{proof}

Proposition \ref{prop:Yk} is rather close to our hope that a homeomorphism sends cosets into cosets, since both $Y_k(P)$ and $Y_k(P')$ are union of cosets. Further improvements of this result hinge on the explicit determination of the sets $Y_k(P)$; however, this will depend closely on the actual structure of the prime ideals and the units of $R$, and in particular on the image of $U(R)$ in $R/P^n$.
\begin{prop}\label{prop:Yk-UR}
Let $R$ be a Dedekind domain with torsion class group, and let $P$ be an almost cyclic prime ideal; let $p:=|R/P|$. Suppose that $U(R)$ is finite. Then, the following hold.
\begin{enumerate}[(a)]
\item\label{prop:Yk-UR:gen} There are $m\geq 0$ and $t\geq 1$ such that, for every $N\geq m$, we have $Y_N(P)=U(R)+P^{N+t}$.
\item\label{prop:Yk-UR:cop} If $|U(R)|$ is coprime with $p$, then we can take $m=t=1$. Furthermore, in this case
\begin{equation*}
\eta(P)=|H_1(P)|=\frac{p-1}{|\pi_1(U(R))|}.
\end{equation*}
\end{enumerate}
\end{prop}
\begin{proof}
\ref{prop:Yk-UR:gen} By Lemma \ref{lemma:almcyc-cardRP}, the cardinality $p$ of $R/P$ is a prime number. 

Since $U(R)$ is finite, we can find $M'$ such that the kernel of the map $\widetilde{\pi}_n:U(R)\longrightarrow H_n(P)$ is equal to the kernel of $\widetilde{\pi}_{M'}$ for every $n\geq M'$. Furthermore, by Lemma \ref{lemma:etaP} there is an $M''$ such that $\eta(P)$ divides $|H_{M''}(P)|$. Take $M:=\max\{M',M''\}$; then, $|H_M(P)|=p^m\eta(P)$ for some $0\leq m<M$, and thus $|H_{M+k}(P)|=p^{m+k}\eta(P)$ for every $k\geq 0$.

By Theorem \ref{teor:corresp-Theta}, $Y_N(P)$ correspond to the subgroup of index $p^N\eta(P)$ in $H_k(P)$, for $k\gg 0$. If $N\geq m$, let $N:=m+k$; then, $|H_{M+k}(P)|=p^N\eta(P)$, and thus $Y_N(P)$ corresponds exactly to the identity subgroup of $H_{M+k}(P)$, i.e., $Y_N=U(R)+P^{M+k}$. However, $M+k=M+N-m$; setting $t:=M-m$ we have our claim.

\ref{prop:Yk-UR:cop} If the cardinality of $U(R)$ is coprime with $p$, then for every $n\geq 1$ the non-$p$-component of $H_1(P)$ and $H_n(P)$ are isomorphic, and the image of $U(R)$ in $H_1(P)$ and $H_n(P)$ is the same; in particular, $|\pi_n(U(R))|=|\pi_1(U(R))|$ and the formula holds.

In particular, with the notation of the previous part of the proof, we have $M'=M''=1$, $m=0$ and $t=1-0=1$. The claim is proved.
\end{proof}

We now restrict to the case $R=\insZ$; we first specialize the previous proposition.
\begin{prop}
Let $p$ be a prime number, and let $k\geq 0$. Then, the following hold.
\begin{enumerate}[(a)]
\item If $p=2$, then $Y_k(2\insZ)=(1+2^{k+2}\insZ)\cup(-1+2^{k+2}\insZ)$.
\item If $p>2$, then $Y_k(p\insZ)=(1+p^{k+1}\insZ)\cup(-1+p^{k+1}\insZ)$
\end{enumerate}
\end{prop}
\begin{proof}
For $p>2$ the claim is exactly the one in Proposition \ref{prop:Yk-UR}\ref{prop:Yk-UR:cop}. For $p=2$, we can take $M=2$, so $m=0$, $t=1$ and thus $Y_k=\pm 1+2^{k+2}\insZ$, as claimed.
\end{proof}

A different way to express the previous proposition is the following.
\begin{prop}
Let $p$ be a prime number, $a$ an integer coprime with $p$, and $k\geq 0$. Then:
\begin{enumerate}[(a)]
\item if $a$ is even, then $a\in Y_k(p\insZ)$ if and only if $p^{k+1}$ divides $a^2-1$;
\item if $a$ is odd, then $a\in Y_k(p\insZ)$ if and only if $p^{k+1}$ divides $\frac{a^2-1}{4}$.
\end{enumerate}
\end{prop}
\begin{proof}
If $a$ is even, then $p$ is odd. Then, $a\in Y_k(p\insZ)$ if and only if $p^{k+1}$ divides $a-1$ or $a+1$. Since $p$ cannot divide $a-1$ and $a+1$ at the same time, this happens if and only if $p^{k+1}$ divides $a^2-1$. 

If $a$ is odd and $p$ is odd, the same reasoning applies (noting that $p^{k+1}$ divides $a^2-1$ if and only if it divides $\frac{a^2-1}{4})$. If $p=2$, then one between $a-1$ and $a+1$ is in the form $2b$ for $b$ odd, while the other is in the form $2^jc$ with $c$ odd and $j\geq 2$. Hence, $a\in Y_k(2\insZ)$ if and only if $j\geq k+2$, i.e., if and only if $2^{k+3}$ divides $a^2-1$. Dividing by $4$ we have our claim.
\end{proof}

For any $n\inZ$, let now
\begin{equation*}
n^\star:=\begin{cases}
n^2-1 & \text{if~}n\text{~is even},\\
\frac{n^2-1}{4} & \text{if~}n\text{~is odd}.
\end{cases}
\end{equation*}
This notation allows to simplify the previous proposition.

\begin{cor}\label{cor:nstar}
Let $h$ be a self-homeomorphism of $G(\insZ)$, and let $n\inZ$ such that $|n|>1$. If $n^\star$ factors as $p_1^{e_1}\cdots p_t^{e_t}$, then $h(n)^\star$ factors as $q_1^{e_1}\cdots q_t^{e_t}$, where $h(p_i\insZ)=q_i\insZ$.
\end{cor}
\begin{proof}
For every $n$, let $X(n)$ be the set of all pairs $(p,k)$ where $p$ is a prime factor of $n^\star$ and $k$ is the largest integer such that $p^{k+1}$ divides $n^\star$. By the previous proposition, $(p,k)\in X(n)$ if and only if $n\in Y_k(p\insZ)$; hence, $X(n)=\{(p_1,e_1-1),\ldots,(p_t,e_t-1)\}$.

Since $h$ is a homeomorphism, $h(Y_k(p_i\insZ))=Y_k(q_i\insZ)$; thus, $X(h(n))=\{(q_1,e_1-1),\ldots,(q_t,e_t-1)\}$. It follows that $h(n)^\star=q_1^{e_1}\cdots q_t^{e_t}$, as claimed.
\end{proof}

Note that the previous corollary is similar to Proposition \ref{prop:DP}, in that it compares the factorization structures of two elements linked by a homeomorphism $h$. However, this result is much more precise, since it applies to every integer (unless only the $\eta(P)$) and, more importantly, the relationship between the corresponding factors $p_i$ and $q_i$ is uniform for every $n$.

\begin{lemma}\label{lemma:star}
Let $n,m\inZ$.
\begin{enumerate}[(a)]
\item If $n$ and $m$ are both even or both odd, then $n^\star=m^\star$ if and only if $|n|=|m|$.
\item If $|n|>1$ and $n^\star$ is prime, then $|n|\in\{2,3\}$.
\end{enumerate}
\end{lemma}
\begin{proof}
The first claim follows directly from the definition. For the second one, since $n^\star=|n|^\star$ we can suppose without loss of generality that $n>0$. If $n>3$ is even, then both $n-1$ and $n+1$ have a (distinct) prime factor, and thus $n^\star=n^2-1=(n-1)(n+1)$ has at least two factors. If $n>3$ is odd, then one between $n-1$ and $n+1$ is divisible by $4$ and the other one by $2$, so that $n^\star$ is even; however, since $n-1>2$, there is at least one odd prime dividing $n-1$ or $n+1$, and thus $n^\star$ has at least two prime factors. The claim is proved.
\end{proof}

\begin{teor}\label{teor:selfOmefZ}
The unique self-homeomorphisms of $G(\insZ)$ are the identity and the multiplication by $-1$.
\end{teor}
\begin{proof}
Let $h:G(\insZ)\longrightarrow G(\insZ)$ be a self-homeomorphism of $\insZ$. We first claim that, for every $n\inZ$, $|h(n)|=n$; we proceed by induction on $n$.

If $|n|=1$ then $n$ is a unit and thus $h(n)\in U(\insZ)=\{\pm 1\}$.

Suppose $|n|=2$. Then, $n^\star=3$, and thus $h(n)^\star$ must be a prime number; by the previous lemma, $h(n)\in\{\pm 2,\pm 3\}$. Suppose that $|h(n)|=3$, so in particular $h(2\insZ)=3\insZ$ and $h(3\insZ)=2\insZ$. Consider $m=7$: then, $m^\star=12=2^2\cdot 3$, and thus by Corollary \ref{cor:nstar} $h(m)^\star$ must be equal to $3^2\cdot 2=18$. Since $h(m)\notin 2\insZ=h(3\insZ)$, we have $m^2=18\cdot 4+1=73$, a contradiction. Hence $h(n)\in\{\pm 2\}$, and at the same time $h(\pm 3)\in\{\pm 3\}$.

Suppose now the claim holds for $|m|<|n|$, with $|n|\geq 4$. In particular, $h(p\insZ)=p\insZ$ for all prime numbers $p$ with $p<|n|$; since $h(2\insZ)=2\insZ$, $n$ and $h(n)$ are either both even or both odd. Let $a:=|n|+1$ and $b:=|n|-1$; then, $n^\star=ab$ or $n^\star=\frac{ab}{4}$ (according to whether $n$ is even or odd). If $a$ is not prime, then all prime factors of $a$ and $b$ are smaller than $|n|$; hence, if $n^\star=p_1^{e_1}\cdots p_n^{e_n}$ by Corollary \ref{cor:nstar} then also $h(n)^\star=p_1^{e_1}\cdots p_n^{e_n}$, and thus $n^\star=h(n)^\star$; by Lemma \ref{lemma:star}, $|n|=|h(n)|$.

Suppose that $a$ is prime: then, $n$ must be even. Hence, $n^\star=(n-1)a$, and by Corollary \ref{cor:nstar} and inductive hypothesis we have $h(n)^\star=(n-1)a'$ for some prime number $a'$. If $|h(n)|\neq|n|$, then $|h(n)|>|n|$ (since all $m$ with $|m|<|n|$ are image of $m$ or $-m$), and since $h(n)$ is even both $|h(n)|-1$ and $|h(n)|+1$ are bigger than $b$. Since $h(n)^\star=(|h(n)|-1)(|h(n)|+1)=(n-1)a'$ and $a'$ is prime, it follows that $a'$ should divide both $|h(n)|-1$ and $|h(n)|+1$, and thus that $a'=2$. However, $h(n)^\star$ is odd; this is a contradiction, and thus $|h(n)|=|n|$.

\medskip

Set now $X:=\{n\inZ\mid h(n)=n\}$ and $Y:=\{n\inZ\mid h(n)=-n\}$: by the previous part of the proof, $X\cup Y=\insZ$, and since $0\notin G(\insZ)$ they are disjoint.

Both sets are closed in $G(\insZ)$: indeed, $X$ is the set of fixed points of $h$, which is closed since $G(\insZ)$ is Hausdorff, while $Y$ is the set of fixed point of $-h$ (i.e., the homeomorphism that sends $n$ to $-h(n)$). Since $G(\insZ)$ is connected \cite[Theoerm 8(b)]{clark-golomb}, they can't be both nonempty: hence, either $X=\emptyset$ (and thus $h$ is the multiplication by $-1$) or $Y=\emptyset$ (and thus $h$ is the identity). The claim is proved.
\end{proof}

\begin{teor}\label{teor:iso-Z}
Let $K$ be an algebraic extension of $\insQ$, and let $R$ be a Dedekind domain with quotient field $K$. If $G(R)\simeq G(\insZ)$, then $R=\insZ$.
\end{teor}
\begin{proof}
By \cite[Theorem 13]{clark-golomb}, the number of units is an invariant of the Golomb topology, and thus $|U(R)|=2$. By Dirichlet's Unit Theorem (see e.g. \cite[Chapter 1, \textsection 7]{neukirch}), $[K:\insQ]\leq 2$. Furthermore, if $R$ is not the ring of integers $\mathcal{O}_K$ of $K$, then there is a prime ideal of $\mathcal{O}_K$ such that $PR=R$; since $\mathcal{O}_K$ has torsion class group, there are elements of $\mathcal{O}_K$ generating a $P\cap\mathcal{O}_K$-primary ideal, and they would be units of $R$, a contradiction. Hence $R=\mathcal{O}_K$.

If $K\neq\insQ$, then there is a prime number $p$ which is inert in $R$; then, $R/pR$ is a field with $p^2$ elements, and by Lemma \ref{lemma:almcyc-cardRP} it follows that $pR$ is not almost cyclic. However, $pR=h(q\insZ)$ for some prime number $q$; since $q\insZ$ is almost cyclic and $\insZ$ is Dirichlet at every prime ideal, this contradicts Corollary \ref{cor:almcyc-topinv}. Hence, $K=\insQ$ and $R=\insZ$, as claimed.
\end{proof}

\bibliographystyle{plain}
\bibliography{/bib/articoli,/bib/libri,/bib/miei}

\begin{thebibliography}{10}

\bibitem{apostol}
Tom~M. Apostol.
\newblock {\em Introduction to analytic number theory}.
\newblock Springer-Verlag, New York-Heidelberg, 1976.
\newblock Undergraduate Texts in Mathematics.

\bibitem{bmt-golomb}
Taras Banakh, Jerzy Mioduszewski, and S{\l}awomir Turek.
\newblock On continuous self-maps and homeomorphisms of the {G}olomb space.
\newblock {\em Comment. Math. Univ. Carolin.}, 59(4):423--442, 2018.

\bibitem{brown-golomb}
Morton Brown.
\newblock A countable connected {H}ausdorff space.
\newblock In L.W. Cohen, editor, {\em The {A}pril meeting in {N}ew {Y}ork},
  volume~4, pages 330--371. Bull. Amer. Math. Soc., 1953.
\newblock Abstract 423.

\bibitem{clark-euclidean}
Pete~L. Clark.
\newblock The {E}uclidean criterion for irreducibles.
\newblock {\em Amer. Math. Monthly}, 124(3):198--216, 2017.

\bibitem{clark-golomb}
Pete~L. Clark, Noah Lebowitz-Lockard, and Paul Pollack.
\newblock A note on {G}olomb topologies.
\newblock {\em Quaest. Math.}, 42(1):73--86, 2019.

\bibitem{dagupta-atlas}
Abhijit Dagupta.
\newblock Countable metric spaces without isolated points.
\newblock In {\em Topology Atlas}. 2005.

\bibitem{davenport-analytic}
Harold Davenport.
\newblock {\em Multiplicative number theory}, volume~74 of {\em Graduate Texts
  in Mathematics}.
\newblock Springer-Verlag, New York, third edition, 2000.
\newblock Revised and with a preface by Hugh L. Montgomery.

\bibitem{furstenberg}
Harry Furstenberg.
\newblock On the infinitude of primes.
\newblock {\em Amer. Math. Monthly}, 62:353, 1955.

\bibitem{gilmer_qr}
Robert Gilmer and Jack Ohm.
\newblock Integral domains with quotient overrings.
\newblock {\em Math. Ann.}, 153:97--103, 1964.

\bibitem{golomb-connectedtop}
Solomon~W. Golomb.
\newblock A connected topology for the integers.
\newblock {\em Amer. Math. Monthly}, 66:663--665, 1959.

\bibitem{golomb-aritmtop}
Solomon~W. Golomb.
\newblock Arithmetica topologica.
\newblock In {\em General {T}opology and its {R}elations to {M}odern {A}nalysis
  and {A}lgebra ({P}roc. {S}ympos., {P}rague, 1961)}, pages 179--186. Academic
  Press, New York; Publ. House Czech. Acad. Sci., Prague, 1962.

\bibitem{knopf}
John Knopfmacher and Stefan Porubsky.
\newblock Topologies related to arithmetical properties of integral domains.
\newblock {\em Exposition. Math.}, 15(2):131--148, 1997.

\bibitem{neukirch}
J\"urgen Neukirch.
\newblock {\em Algebraic {N}umber {T}heory}, volume 322 of {\em Grundlehren der
  Mathematischen Wissenschaften [Fundamental Principles of Mathematical
  Sciences]}.
\newblock Springer-Verlag, Berlin, 1999.
\newblock Translated from the 1992 German original and with a note by Norbert
  Schappacher, With a foreword by G. Harder.

\bibitem{sierpinski-Q}
Wac{\l}aw Sierpi\'{n}ski.
\newblock Sur une propri\'{e}t\'{e} topologique des ensembles denombrables
  denses en soi.
\newblock {\em Fund. Math.}, 1:11--16, 1920.

\end{thebibliography}
\end{document}